\documentclass[12pt]{amsart}
\usepackage{amscd,amsmath,amsthm,amssymb}
\usepackage[left]{lineno}
\usepackage{color}
\usepackage{stmaryrd}
\usepackage[utf8]{inputenc}
\usepackage{cleveref}
\usepackage{epstopdf}
\usepackage{graphicx}
\usepackage{xcolor}
\usepackage{subeqnarray}

\usepackage{cases}

\usepackage{tikz}

\definecolor{verylight}{gray}{0.97}
\definecolor{light}{gray}{0.9}
\definecolor{medium}{gray}{0.85}
\definecolor{dark}{gray}{0.6}

\def\NZQ{\mathbb}               

\def\KK{{\NZQ K}}

 \def\ab{{\mathbf a}}
 \def\xb{{\mathbf x}}

\def\KK{{\NZQ K}}

\def\G{{\mathcal G}}

\def\C{{\mathcal C}}

\def\ab{{\mathbf a}}
\def\bb{{\mathbf b}}
\def\xb{{\mathbf x}}

\def\cb{{\mathbf c}}
\def\db{{\mathbf d}}

\def\0b{{\mathbf 0}}
\def\alphab{{\mathbf \alphab}}
\def\c_ib{{\mathbf c_i}}

\def\reg{{\mathbf reg}}
\def\height{\operatorname{ht}}
\def\depth{\operatorname{depth}}
\def\opn#1#2{\def#1{\operatorname{#2}}} 
\opn\chara{char} \opn\length{\ell} \opn\pd{pd} \opn\rk{rk}
\opn\projdim{proj\,dim} \opn\injdim{inj\,dim} \opn\rank{rank}
\opn\depth{depth} \opn\grade{grade} \opn\height{height}
\opn\embdim{emb\,dim} \opn\codim{codim}

\opn\Tr{Tr} \opn\bigrank{big\,rank}
\opn\superheight{superheight}\opn\lcm{lcm}
\opn\trdeg{tr\,deg}
\opn\reg{reg} \opn\lreg{lreg} \opn\ini{in} \opn\lpd{lpd}
\opn\size{size} \opn\sdepth{sdepth}
\opn\link{link}\opn\fdepth{fdepth}\opn\lex{lex}
\opn\tr{tr}
\opn\type{type}
\opn\gap{gap}
\opn\arithdeg{arith-deg}
\opn\HS{HS}
\opn\GL{GL}
\opn\div{div} \opn\Div{Div} \opn\cl{cl} \opn\Cl{Cl}
\opn\Spec{Spec} \opn\Supp{Supp} \opn\supp{supp} \opn\Sing{Sing}
\opn\Ass{Ass} \opn\Min{Min}\opn\Mon{Mon}
\opn\Ann{Ann} \opn\Rad{Rad} \opn\Soc{Soc}\opn\Deg{Deg}
\opn\Im{Im} \opn\Ker{Ker} \opn\Coker{Coker} \opn\Am{Am}
\opn\Hom{Hom} \opn\Tor{Tor} \opn\Ext{Ext} \opn\End{End}
\opn\Aut{Aut} \opn\id{id}

\opn\nat{nat}
\opn\pff{pf}
\opn\Pf{Pf} \opn\GL{GL} \opn\SL{SL} \opn\mod{mod} \opn\ord{ord}
\opn\Gin{Gin} \opn\Hilb{Hilb}\opn\sort{sort}
\opn\PF{PF}\opn\Ap{Ap}
\opn\mult{mult}
\opn\bight{bight}
\opn\aff{aff}
\opn\relint{relint} \opn\st{st}
\opn\lk{lk} \opn\cn{cn} \opn\core{core} \opn\vol{vol}  \opn\inp{inp} \opn\nilpot{nilpot}
\opn\link{link} \opn\star{star}\opn\lex{lex}\opn\set{set}
\opn\width{wd}
\opn\Fr{F}
\opn\QF{QF}
\opn\G{G}
\opn\type{type}\opn\res{res}
\opn\conv{conv}
\opn\Ind{Ind}
\opn\gr{gr}
\def\Rees{{\mathcal R}}
\def\pot#1#2{#1[\kern-0.28ex[#2]\kern-0.28ex]}

\opn\dirlim{\underrightarrow{\lim}}
\opn\inivlim{\underleftarrow{\lim}}
\def\Implies{\ifmmode\Longrightarrow \else
	\unskip${}\Longrightarrow{}$\ignorespaces\fi}
\def\implies{\ifmmode\Rightarrow \else
	\unskip${}\Rightarrow{}$\ignorespaces\fi}
\def\iff{\ifmmode\Longleftrightarrow \else
	\unskip${}\Longleftrightarrow{}$\ignorespaces\fi}

\let\:=\colon
\newtheorem{Theorem}{Theorem}[section]
\newtheorem{Lemma}[Theorem]{Lemma}
\newtheorem{Corollary}[Theorem]{Corollary}

\newtheorem{Remark}[Theorem]{Remark}

\newtheorem{Definition}[Theorem]{Definition}

\let\epsilon\varepsilon
\let\kappa=\varkappa
\def\qed{\ifhmode\textqed\fi
	\ifmmode\ifinner\quad\qedsymbol\else\dispqed\fi\fi}
\def\textqed{\unskip\nobreak\penalty50
	\hskip2em\hbox{}\nobreak\hfil\qedsymbol
	\parfillskip=0pt \finalhyphendemerits=0}
\def\dispqed{\rlap{\qquad\qedsymbol}}

\opn\dis{dis}
\def\pnt{{\raise0.5mm\hbox{\large\bf.}}}

\opn\Lex{Lex}

\newcommand*{\circled}[1]{\lower.7ex\hbox{\tikz\draw (0pt, 0pt)%
		circle (.5em) node {\makebox[1em][c]{\small #1}};}}

\begin{document}
	\title {Integral closure and normality of edge ideals of some edge-weighted graphs}
	
	\author {Shiya Duan, Guangjun Zhu$^{\ast}$, Yijun Cui and Jiaxin Li}

	\address{School of Mathematical Sciences, Soochow University, Suzhou, Jiangsu, 215006, P. R. China}

	\email{3136566920@qq.com(Shiya Duan), zhuguangjun@suda.edu.cn(Corresponding author:Guangjun Zhu),
		237546805@qq.com(Yijun Cui),lijiaxinworking@163.com(Jiaxin Li).}
	
	\thanks{$^{\ast}$ Corresponding author}



	
	
	\maketitle
\begin{abstract}
Let $G_\omega$ be an edge-weighted simple graph. In this paper, we  give a  complete  characterization of the graph $G_\omega$ whose edge ideal $I(G_\omega)$ is integrally closed.
We also show that if  $G_\omega$ is an edge-weighted star graph,  a  path or a cycle, and $I(G_\omega)$ is integrally closed, then $I(G_\omega)$ is normal.
\end{abstract}
	\setcounter{tocdepth}{1}

    \section{Introduction}
Let $S=\KK[x_{1},\dots, x_{n}]$ be a polynomial ring in $n$ variables over a field $\KK$. The class of monomial ideals of $S$ has been intensively studied and many problems arise in when we would like to study good properties of monomial ideals, such as the integral closure and normality. Recall that an ideal $I\subset S$ is called {\em integral closure} if $I=\overline{I}$ (see Definition \ref{integrally closed_1}
 for the  exact  definitions of $\overline{I}$),  and $I$ is called {\em normal} if $I^i=\overline{I^i}$  for all $i\ge 1$. This notion is  related to the graded algebras arising from $I$ such as the Rees algebra $\Rees(I)=\oplus_{i\ge 0} I^i t^i $. It is known that  $I$ is normal if and only if $\Rees(I)$ is normal, see \cite[Theorem 4.3.17]{V3}.  This  highlights the importance of studying the normality of ideals.  It is well-known that every square-free monomial ideal is integrally closed, see \cite[Theorem 1.4.6]{HH}. Appearing as edge and cover ideals of graphs, the square-free monomial ideals play a key role in the connection between commutative algebra and combinatorics, see \cite{EVY,V1}.
The normality of such ideals has been of interest to many authors, see \cite{RV,V1,V2}.  For example, in \cite{SVV} it is shown  that the edge ideals of bipartite graphs are normal. And it is also shown in \cite[Corollary 14.6.25]{V3} that the cover ideals of perfect graphs are normal.  In \cite{ANR} it is shown that the cover ideals of odd cycles and wheel graphs are normal.

Let $G$ be a simple graph with vertex set $V(G)=[n]$ and edge set $E(G)$,  where  $[n]$ is by convention  the set $\{1,2,\ldots, n\}$. Let $G_\omega$ be
an {\em edge-weighted} (or simply weighted)  graph whose underlying graph is $G$, that is, $G_\omega$ is a triplet $(V(G_\omega), E(G_\omega), w)$
where $V(G_\omega)=V(G)$, $E(G_\omega)=E(G)$  and $w: E(G_\omega)\rightarrow \mathbb{N}^{+}$ is a weight function. Here $\mathbb{N}^{+}$ denotes the set of positive
integers. We often  write $G_\omega$ for the triplet $G_\omega=(V(G_\omega), E(G_\omega),\omega)$. In other words, $G_\omega$ is obtained from $G$ by assigning  a weight to its edges.
An edge-weighted  graph is  called a \emph{non-trivially  weighted} graph if there is at least one edge with  a weight  greater than $1$. Otherwise, it is called a
{\em trivially weighted} graph.  We consider the polynomial ring  $S=\KK[x_{1},\dots, x_{n}]$   in $n$ variables over a field $\KK$.
 The  {\em edge-weighted ideal}  (or simply edge ideal)  of   $G_\omega$, was introduced in \cite{PS}, is  the ideal of $S$ given by
$$I(G_\omega)=(x_i^{\omega(e)}x_j^{\omega(e)}\mid e:=\{i,j\}\in E(G_\omega)).$$
If $G_\omega$ is trivially weighted, then  $I(G_\omega)$ is the usual edge ideal of underlying graph $G$ of $G_\omega$,  that
has been extensively studied in the literature \cite{GV, HH,MV,VT,V3,Zhu1,Zhu2}.

Paulsen and Sather-Wagstaff in \cite{PS} studied the primary decomposition of these ideals. They
also studied the unmixedness and Cohen-Macaulayness of these ideals, in the case where $G_\omega$ is a
cycle, a tree, or a complete graph. In \cite{FSTY},  Seyed Fakhari et al.  characterize the unmixedness and Cohen-Macaulayness of edge-weighted ideals of very
well-covered graphs.

Little is known about integral closure and normality of edge ideals of edge-weighted graphs.  In this paper, we aim to characterise a weighted graph  $G_\omega$ whose edge ideal $I(G_\omega)$ is integrally closed. Under the condition that $I(G_\omega)$ is integrally closed,
we  show that $I(G_\omega)$ is normal if  $G_\omega$ is a weighted star graph or a path or a cycle.

The paper is organized as follows. In Section $2$, we recall some essential definitions and
terminology that we will need later. In Section $3$, we  give a complete characterization of a weighted graph $G_\omega$ whose edge ideal $I(G_\omega)$ is integrally closed.
Under the condition that $I(G_\omega)$  is integrally closed, we show in  Section $4$  that  $I(G_\omega)$  is normal if $G_\omega$ is a weighted star graph or a path or a cycle.

\section{Preliminary}
In this section, we gather together the needed definitions and basic facts that will
be used throughout this paper. However, for more  details, we refer the reader to \cite{BM,HH,HS,PS,V}.

A weighted graph $H_\omega=(V(H), E(H),\omega)$ is called an \emph{induced subgraph} of a weighted graph   $G=(V(G), E(G),\omega)$ if $V(H)\subset V(G)$,  for any $u,v\in V(H)$, $\{u,v\}\in E(H)$  if and only if  $\{u,v\}\in E(G)$, and its weight $\omega_H(\{u,v\})$ in $H$ is equal to its weight $\omega_G(\{u,v\})$ in $G$. For  convenience, we  call $H_\omega$ an induced subgraph of $G_\omega$.  For $A\subset V(G)$, let $G[A]$ denote the \emph{induced subgraph} of $G$ on the  set $A$. 

 A connected weighted graph  $G_\omega$  is called a cycle if $\deg_G(v)=2$ for all $v\in  V(G)$. A cycle with $n$
vertices is called to be an $n$-cycle,  denoted by $C_\omega^n$. A connected weighted graph on  the set $[n]$ is  called  a path, if
$E(G)=\{\{i,i+1\} | 1\le i \le n-1\}$. Such a path is usually denoted by $P_\omega^n$.
 A weighted simple graph $G_\omega$ on vertex set $[n]$ is called to be a complete graph, if $\{i,j\} \in E(G)$ for all $i,j \in [n]$. A complete graph with $n$  vertices is usually denoted by $K_\omega^n$. A weighted graph $G_\omega$ is chordal if every induced cycle in $G_\omega$ is a $3$-cycle $C_\omega^3$.

 \begin{Definition}{\em (\cite[Definition 1.4.1]{HH})}\label{integrally closed_1} Let $R$ be a ring and $I$ an ideal in $R$.
    An element $f \in R$ is  said to be \emph{integral} over $I$, if there exists an equation
    \[
    f^k+c_1f^{k-1}+\dots+c_{k-1}f+c_k=0 \text{\ \ with\ \ }c_i \in I^i.
    \]
    The set $\overline{I}$  of elements in $R$ which are integral over $I$ is the \emph{integral closure}
    of $I$. The ideal $I$ is {\em integrally closed}, if $I=\overline{I}$, and $I$ is \emph{normal} if all powers
    of $I$ are integrally closed.
    \end{Definition}

   For an ideal $I$ in $R$, it is clear that $I\subseteq \overline{I}$, so $I$ is integrally closed if and only if $\overline{I}\subseteq I$. Further, if $I$ is a monomial ideal, then $\overline{I}$
    can be described as follows:
    \begin{Theorem}{\em (\cite[Theorem 1.4.2]{HH})}\label{integrally closed_2}
    Let $I \subset S$ be a monomial ideal. Then $\overline{I}$ is a monomial ideal generated by all monomials $f \in S$ for which there exists an integer $k$ such that $f^k \in I^k$.
    \end{Theorem}
    According to Theorem \ref{integrally closed_2}, we have another description of the integral closure of $I$:
   \[
   \overline{I}=(f\in S\mid  f \text{\ is a monomial  and $f^i \in I^i$ for some $i\geq 1$}).
\]

 Let $G$  be a simple graph with the vertex set $V(G)=[n]$ and the edge set $E(G)$, where  for  convention the notation $[n]$ denotes the
set $\{1,\ldots,n\}$. The \emph{neighbourhood} of a vertex $v$ in $G$ is defined as $N_G(v)=\{u\in V (G) : \{u,v\}\in  E(G)\}$ and
its degree, denoted by $\deg_G(v)$, is $|N_G(v)|$.

\medskip
For a monomial $u=x_1^{a_1}\cdots x_n^{a_n}\in S$, we denote by $\Gamma(u)$ the exponent vector
  $(a_1,\ldots,a_n)$ of $u$.  In this case, we can write $u$ as $u=\xb^{\ab}$ with $\ab=(a_1,\ldots,a_n)\in \mathbb{Z}_{+}^n$.
  Observe that there exists a bijection which takes a monomial $u=x_1^{a_1}\cdots x_n^{a_n}$ into
a  vector $(a_1,\ldots,a_n)$ in $\mathbb{Z}_{+}^n$, where $\mathbb{Z}_{+}^n$ is the set of those vectors $(a_1,\ldots,a_n)\in \mathbb{Z}^n$ with each $a_i\ge 0$. Similarly, if $A$ is a set of monomials in $S$ we set $\Gamma(A)=\{\Gamma(u): u\in A\}$. For a monomial ideal $I\subset S$, let  $\mathcal{G}(I)$ denote the minimal set of generators of its monomial. If $\mathcal{G}(I)=\{\xb^{{\bb}_1},\ldots, \xb^{{\bb}_m}\}$,   then
we denote the convex hull of $\Gamma(\mathcal{G}(I))$ by $\C(I)$, i.e., $\C(I)=\{\ab\in \mathbb{Q}_{+}^n \mid \ab\in conv({\bb}_1,\dots,{\bb}_m)\}=\{\ab=\sum\limits_{i=1}\limits^{m}{\lambda_i{\bb}_i} \mid \sum\limits_{i=1}\limits^{m}{\lambda_i}=1, \lambda_i \in \mathbb{Q}_{+}\}$, where $\mathbb{Q}_{+}$ is the set of all nonnegative rational numbers. We call  $\C(I)$  the {\em Newton polyhedron} of $I$.

  \begin{Lemma}{\em (\cite[Proposition 12.1.4]{V3})}\label{integrally closed_3}
   Let $I\subset S$ be a monomial ideal with $\mathcal{G}(I)=\{\xb^{{\bb}_1},\ldots, \xb^{{\bb}_m}\}$. Then $\overline{I}$ is generated by the
monomials  $\xb^{\ab}$, where $\ab=(\lceil a_1 \rceil,\dots,\lceil a_n \rceil)$ with $(a_1,\ldots,a_n)\in \C(I)$ and   each $\lceil a_i \rceil$ is the smallest integer $\ge a_i$.
    \end{Lemma}

    \section{Integral closure of edge ideals of   edge-weighted graphs}

    In this section, we will give a characterization of weighted graphs whose edge ideals are integrally closed.

        \begin{Theorem}
    	\label{trivial}
    	Let $G_\omega=(V(G_\omega), E(G_\omega))$ be a weighted graph with at most one edge  having non-trivial weight, then $I(G_\omega)$  is integrally closed.
    \end{Theorem}

    \begin{proof}
    	Let $E(G_\omega)=\{e_1,\ldots, e_m\}$, where $e_i=\{u_i,v_i\}$ and  $\omega_i=\omega{(e_i)}$. Without loss of generality, we assume that $\omega_1\geq 1$ and $\omega_i=1$ for all $i=2,\dots,m$.
    Let $\xb^{{\bb}_i}=x_{u_i}^{\omega_i}x_{v_i}^{\omega_i}$ for   $i=1,\ldots, m$, then
     ${\bb}_i=(0,\dots,0,\omega_i,0,\dots,0,\omega_i,0,\dots,0)$ where  $\omega_i$  are the $u_i$-th and $v_i$-th entries of ${\bb}_i$ respectively.
By Lemma \ref{integrally closed_3}, we have
  \[
     \mathcal{G}(\overline{I(G_\omega)})=\{\xb^{\ab}\mid \ab=(\lceil a_1 \rceil,\dots,\lceil a_n \rceil) \ \text{with\ } (a_1,\ldots,a_n)\in  \C(I(G_\omega))\}.
\]
Let $\xb^{\ab}\in \mathcal{G}(\overline{I(G_\omega)})$ with  $\ab=(\lceil a_1 \rceil,\dots,\lceil a_n \rceil)$ satisfying $(a_1,\ldots,a_n)=\sum\limits_{i=1}\limits^{m}{\lambda_i{\bb}_i}$ with
$\sum\limits_{i=1}\limits^{m}{\lambda_i}=1$ and $\lambda_i \in \mathbb{Q}_{+}$.
If $\lambda_{i}=1$ and $\lambda_{j}=0$ for any $j\in [m]$ with $j\ne i$, then $\xb^{\ab}=\xb^{\lambda_i{\bb}_i}\in I(G_\omega)$.
If there exists some $t\ge 2$  such that  $\lambda_{i_1},\ldots, \lambda_{i_t}>0$. Let $i_1<\cdots<i_t$, then $i_2\ge 2$. Since  $\omega_i=1$ for each $i=2,\dots,m$,  $\xb^{{\bb}_{i_2}}|\xb^{\ab}$. It follows that $\xb^{\ab}\in I(G_\omega)$.
\end{proof}

\begin{Lemma}
\label{subgraph}
 Let $G_\omega$ be a weighted graph and $H_\omega$  its induced subgraph. If for some $k\in \mathbb{N}$, $I(G_\omega)^{k}$ is integrally closed, then $I(H_\omega)^{k}$ is  also integrally closed.
    \end{Lemma}
\begin{proof}
For any monomial $f \in \mathcal{G}(\overline{I(H_\omega)^{k}})$,  we first prove that if $x_u$ divides $f$
 then $u$ is  in $V(H_\omega)$. Indeed, by the choice of $f$
one has $f^s\in  I(H_\omega)^{sk}$ for some  integer $s\ge 1$ by Theorem \ref{integrally closed_2}. So we can write   $f^{s}=h\prod\limits_{i=1}^{sk}(x_{u_i}x_{v_i})^{\omega(e_i)}$ for some monomial   $h$ and $e_{i}=\{u_i,v_i\}\in E(H_\omega)$ for $i=1,\ldots, sk$.
Since  $x_u|f$, we have   $x_u^{s}|f^{s}$. If  $u\notin V(H_\omega)$, then $u\cap e_i=\emptyset$ for $i=1,\ldots, sk$. This forces that $x_u^{s}|h$.  Let $h=x_u^{s}h_1$, then $f^{s}=x_u^{s}h_1\prod\limits_{i=1}^{sk}(x_{u_i}x_{v_i})^{\omega(e_i)}$. So $(f/x_u)^s=h_1\prod\limits_{i=1}^{sk}(x_{u_i}x_{v_i})^{\omega(e_i)}$, which implies  that $f/x_u \in \overline{I(H_\omega)^{k}}$  by Theorem \ref{integrally closed_2}. This contradicts the fact that $f \in \mathcal{G}(\overline{I(H_\omega)^{k}})$.

Since $f \in \mathcal{G}(\overline{I(H_\omega)^{k}})$, one has $f\in  \overline{I(G_\omega)^{k}}$ by \cite[Remark 1.1.3]{HS}. So $f\in I(G_\omega)^{k}$, since $I(G_\omega)^{k}$ is integrally closed.
 It follows that $f=h\prod\limits_{i=1}^{k}(x_{u_i}x_{v_i})^{\omega(e_i)}$ for some monomial   $h$ and $e_{i}=\{u_i,v_i\}\in E(G_\omega)$ for $i=1,\ldots, k$.
By  the above proof, we get $u_i,v_i\in V(H_\omega)$ and $e_{i}=\{u_i,v_i\}\in E(H_\omega)$.   Consequently, $f\in I(H_\omega)^{k}$. This completes our proof.
 \end{proof}

  \begin{Remark}
\label{induced graph}
 Let $G_\omega$ be a weighted graph and $H_\omega$ be its  induced subgraph. If $I(G_\omega)$ is
 normal then $I(H_\omega)$ is also  normal.
\end{Remark}

The next lemma gives a list of weighted  graphs which are not integrally closed.
    \begin{Lemma} 		
    	\label{3subgraph}
   Let $G_\omega$ be a non-trivially weighted graph, such that all of its edges have non-trivial weights.
  \begin{enumerate}
    \item \label{3subgraph-1}If $G_\omega=P_\omega^3$ is a path of length $2$, then $I(G_\omega)\neq \overline{I(G_\omega)}$.
    \item If $G_\omega=P_\omega^2\sqcup P_\omega^2$ is a disjoint union of two paths $P_\omega^2$,  then $I(G_\omega)\neq \overline{I(G_\omega)}$.
    \item If $G_\omega=C_\omega^3$ is a $3$-cycle, then $I(G_\omega)\neq \overline{I(G_\omega)}$.
\end{enumerate}
\end{Lemma}
    \begin{proof} (1) Let  $V(G_\omega)=[3]$ and $E(G_\omega)=\{\{1,2\},\{2,3\}\}$, then $\mathcal{G}(I(G_\omega))=\{x_1^{\omega_1}x_2^{\omega_1},\\
    x_2^{\omega_2}x_3^{\omega_2}\}$ with each $\omega_i=\omega(\{i,i+1\})$. Choose  $f=x_1^{\omega_1-1}x_2^{\omega_1+\omega_2}x_3^{\omega_2-1}$, then  $f\notin I(G_\omega)$, but $f^2=x_1^{2\omega_1-2}x_2^{2\omega_1+2\omega_2}x_3^{2\omega_2-2}=(x_1^{\omega_1}x_2^{\omega_1})(x_2^{\omega_2}x_3^{\omega_2})x_2^{\omega_1+\omega_2}x_1^{\omega_1-2}x_3^{\omega_2-2}
    \in I(G_\omega)^2$. This means that $f\in \overline{I(G_\omega)}$ by Theorem \ref{integrally closed_2}.

   (2) Let $V(G_\omega)=[4]$ and $E(G_\omega)=\{\{1,2\},\{3,4\}\}$,  then $\mathcal{G}(I(G_\omega))=\{x_1^{\omega_1}x_2^{\omega_1},x_3^{\omega_3}x_4^{\omega_3}\}$, where $\omega_1=\omega(\{1,2\})$ and
   $\omega_3=\omega(\{3,4\})$. Choose $g=x_1^{\omega_1-1}x_2^{\omega_1-1}x_3^{\omega_3-1}x_4^{\omega_3-1}$, then $g\notin I(G_\omega)$, but $g^2=x_1^{2\omega_1-2}x_2^{2\omega_1-2}x_3^{2\omega_3-2}x_4^{2\omega_3-2}=(x_1^{\omega_1}x_2^{\omega_1})(x_3^{\omega_3}x_4^{\omega_3})(x_1x_2)^{\omega_1-2}(x_3x_4)^{\omega_3-2}\\
   \in I(G_\omega)^2$. This implies that
   $g\in \overline{I(G_\omega)}$ by Theorem \ref{integrally closed_2}.

   (3) Let $V(G_\omega)=[3]$ and $E(G_\omega)=\{\{1,2\},\{2,3\},\{3,1\}\}$, then $\mathcal{G}(I(G_\omega))=\{x_1^{\omega_1}x_2^{\omega_1},x_2^{\omega_2}x_3^{\omega_2},x_3^{\omega_3}x_1^{\omega_3}\}$, where $\omega_i=\omega(\{i,i+1\})$ and $i+1\equiv j \text{ mod\ } 3$ with $0<j\le 3$ for $i=1,2,3$.
If $\omega_1>\omega_2-1$, we choose $h=x_1^{\omega_3-1}x_2^{\omega_2-1}x_3^{\omega_2+\omega_3}$, It is clear that  $h\notin I(G_\omega)$, but
    $h^2=(x_1^{\omega_3}x_3^{\omega_3})(x_2^{\omega_2}x_3^{\omega_2})x_1^{\omega_3-2}x_2^{\omega_2-2}x_3^{\omega_3+\omega_2}\in I(G_\omega)^2$. Otherwise, we choose $h=x_1^{\omega_1+\omega_3}x_2^{\omega_1-1}x_3^{\omega_3-1}$. In this case, we get that  $h\notin I(G_\omega)$, but $h^2=(x_1^{\omega_1}x_2^{\omega_1})(x_1^{\omega_3}x_3^{\omega_3})x_1^{\omega_1+\omega_3}x_2^{\omega_1-2}x_3^{\omega_3-2}\in I(G_\omega)^2$.
Therefore, by Theorem \ref{integrally closed_2}, we get that $h\in \overline{I(G_\omega)}$.
    \end{proof}

 \begin{Corollary}
     \label{not closed}
      Let  $G_\omega$ be a weighted  graph. If  $G_\omega$ contains one of the three graphs described in Lemma \ref{3subgraph} as an induced subgraph, then $I(G_\omega)$ is not integrally closed.
    \end{Corollary}
 \begin{proof}
    Let $H_\omega$ be an induced subgraph of $G_\omega$ as described  in Lemma \ref{3subgraph}, then $I(H_\omega)\neq\overline{I(H_\omega)}$ by Lemma \ref{3subgraph}. The desired result follows from  Lemma \ref{subgraph}.
    \end{proof}

    \begin{Theorem}
    \label{main}
    	Let $G_\omega$ be a weighted graph. Then $I(G_\omega)$ is  integrally closed if and only if  $G_\omega$ does not contain one of the three graphs described in Lemma \ref{3subgraph} as an induced subgraph.
    \end{Theorem}
 \begin{proof} Necessity  follows from  Corollary \ref{not closed}.
For sufficiency, suppose that $G_\omega$ does not contain any of the three graphs described in Lemma \ref{3subgraph} as its induced subgraph.
 Let $E(G_\omega)=\{e_1,\dots,e_m\}$ with each  $e_i=\{u_i, v_i\}$ and $\omega_i=\omega(e_i)$.
If $G_\omega$ has at most one edge with non-trivial weight, then $I(G_\omega)$ is  integrally closed by Theorem \ref{trivial}. Now suppose  that  $G_\omega$ has $p$ edges with non-trivial
 weights, where $p\ge 2$.
 Without loss of generality, we assume that  $\omega_i\ge 2$  for  $i=1,\ldots,p$ and $\omega_i=1$ for $i=p+1,\dots,m$.
 Set $\xb^{{\bb}_i}=x_{u_i}^{\omega_i}x_{v_i}^{\omega_i}$ for   $i=1,\ldots, m$, then the  exponent vector
     ${\bb}_i=(0,\dots,0,\omega_i,0,\dots,0,\omega_i,0,\dots,0)$,  where  $\omega_i$  are the $u_i$-th and $v_i$-th entries of ${\bb}_i$, respectively.
By Lemma \ref{integrally closed_3}, we have
\[
     \mathcal{G}(\overline{I(G_\omega)})=\{\xb^{\ab}\mid \ab=(\lceil a_1 \rceil,\dots,\lceil a_n \rceil) \ \text{with\ } (a_1,\ldots,a_n)\in  \C(I(G_\omega))\}.
\]
Let $\xb^{\ab}\in \mathcal{G}(\overline{I(G_\omega)})$ with  $\ab=(\lceil a_1 \rceil,\dots,\lceil a_n \rceil)$ satisfying
 \begin{align}
(a_1,\ldots,a_n)=\sum\limits_{i=1}^{m}{\lambda_i{\bb}_i}\ \text{ with \ }
\sum\limits_{i=1}^{m}{\lambda_i}=1\  \text{ and \ } \lambda_i \in \mathbb{Q}_{+}.\label{eqn:SES-1}
 \end{align}
We will prove that $\xb^{\ab }\in I(G_\omega)$. We distinguish into the following two cases:

(i) If  $\lambda_{i}=1$ and $\lambda_{\ell}=0$ for any $\ell\in [m]$ with $\ell\ne i$ in the above expression (\ref{eqn:SES-1}) of $(a_1,\ldots,a_n)$,  then $\xb^{\ab }=\xb^{{\bb}_i}\in I(G_\omega)$.

(ii) If $\lambda_{i_1},\ldots, \lambda_{i_t}>0$ with  $t\ge 2$  in the  expression (\ref{eqn:SES-1}) of $(a_1,\ldots,a_n)$. In this case,
we consider the following two cases:

(a) If $\lambda_{i_\ell}>0$ with $i_\ell>p$, then  $\xb^{{\bb}_{i_\ell}}|\xb^{\ab}$, since $\omega_i=1$ for $i=p+1,\dots,m$. This implies that $\xb^{\ab}\in I(G_\omega)$.

(b)  If $\{i_1,\dots,i_t\} \subseteq \{1,\dots,p\}$. Without loss of generality, we can assume that $\lambda_i>0$ for all $i\in [t]$.
In this case, let $H_\omega$ be an induced subgraph of $G_\omega$  on the  set $A$, where $A=\{u_1,v_1\}\cup \{u_2,v_2\}$.

If $|E(H_\omega)|=2$, then $H_\omega=P_\omega^3$ or $H_\omega=P_\omega^2\sqcup P_\omega^2$ is a disjoint union of two paths $P_\omega^2$. In both cases, every edge of $H_\omega$ has non-trivial weight, which  contradicts the assumption that $G_\omega$ does not contain $P_\omega^3$ or $P_\omega^2\sqcup P_\omega^2$  as its induced subgraph. Consequently,  $|E(H_\omega)|\ge 3$.
 $H_\omega$ can  be only one of the following six cases:

	\vspace{0.7cm}
	\begin{center}
		\begin{tikzpicture}[thick,>=stealth]
		\setlength{\unitlength}{1mm}
		\setlength{\unitlength}{1mm}
		
		\thicklines
		
		\put(-105,92){$u_1$}
		\put(-100,90){\circle*{1.5}}
		
		\put(-79,91){$u_2$}
		\put(-80,90){\circle*{1.5}}
		
		\put(-88,69){$v_1=v_2$}
		\put(-90,70){\circle*{1.5}}	
		\put(-99,80){$e_1$}
		\put(-83,80){$e_2$}
		
		\draw[solid](-10,9)--(-8,9);	
		\draw[solid](-10,9)--(-9,7);
	    \draw[solid](-8,9)--(-9,7);
	
	    \put(-55,70){$v_1$}
	    \put(-50,70){\circle*{1.5}}
	
	    \put(-28,70){$v_2$}
	    \put(-30,70){\circle*{1.5}}
	
	    \put(-52,92){$u_1$}
	    \put(-50,90){\circle*{1.5}}	
	
	    \put(-30,92){$u_2$}
	    \put(-30,90){\circle*{1.5}}	
	    \put(-54,80){$e_1$}
	    \put(-28,80){$e_2$}
	
	    \draw[solid](-5,9)--(-5,7);	
	    \draw[solid](-3,9)--(-3,7);
	    \draw[solid](-3,9)--(-5,9);

        \put(-5,92){$u_1$}
		\put(0,90){\circle*{1.5}}
		
		\put(20,92){$u_2$}
		\put(20,90){\circle*{1.5}}
		
		\put(-5,70){$v_1$}
		\put(0,70){\circle*{1.5}}
		
		\put(22,70){$v_2$}
		\put(20,70){\circle*{1.5}}
		\put(-5,80){$e_1$}
		\put(21,80){$e_2$}

		\draw[solid](0,9)--(2,9);	
		\draw[solid](2,7)--(2,9);
		\draw[solid](0,9)--(0,7);
		\draw[solid](0,7)--(2,7);
		
	\end{tikzpicture}
\end{center}

\vspace{0.2cm}

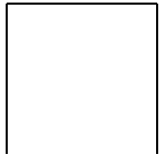
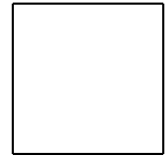
\hspace{0.5cm}(1) $C_\omega^3$ with $\omega_i \geq 2$ \hspace{1.5cm}(2) $P^4_\omega$ with $\omega_i \geq 2$ \hspace{1.5cm}(3) $C_\omega^4$ with $\omega_i \geq 2$

\hspace{1.2cm}for $i=1,2$ \hspace{3.0cm} for $i=1,2$ \hspace{2.7cm} for $i=1,2$

\vspace{1.0cm}

	\begin{center}
		\begin{tikzpicture}[thick,>=stealth]
		\setlength{\unitlength}{1mm}
		\setlength{\unitlength}{1mm}
		
		\thicklines

		\put(-105,70){$v_1$}
		\put(-100,70){\circle*{1.5}}
		
		\put(-78,70){$v_2$}
		\put(-80,70){\circle*{1.5}}
		
		\put(-102,92){$u_1$}
		\put(-100,90){\circle*{1.5}}	
		
		\put(-78,92){$u_2$}
		\put(-80,90){\circle*{1.5}}	
		\put(-104,80){$e_1$}
		\put(-79,80){$e_2$}
		
		\draw[solid](-10,9)--(-10,7);	
		\draw[solid](-8,9)--(-8,7);
		\draw[solid](-8,9)--(-10,9);
		\draw[solid](-8,7)--(-10,9);

        \put(-55,92){$u_1$}
		\put(-50,90){\circle*{1.5}}
		
		\put(-30,92){$u_2$}
		\put(-30,90){\circle*{1.5}}
		
		\put(-56,70){$v_1$}
		\put(-50,70){\circle*{1.5}}
		
		\put(-28,70){$v_2$}
		\put(-30,70){\circle*{1.5}}
		\put(-55,80){$e_1$}
		\put(-29,80){$e_2$}

		\draw[solid](-5,9)--(-3,9);	
		\draw[solid](-3,7)--(-3,9);
		\draw[solid](-5,9)--(-5,7);
		\draw[solid](-5,7)--(-3,7);
		\draw[solid](-5,9)--(-3,7);

        \put(-5,92){$u_1$}
		\put(0,90){\circle*{1.5}}
		
		\put(20,92){$u_2$}
		\put(20,90){\circle*{1.5}}
		
		\put(-5,70){$v_1$}
		\put(0,70){\circle*{1.5}}
		
		\put(21,70){$v_2$}
		\put(20,70){\circle*{1.5}}
		\put(-5,80){$e_1$}
		\put(21,80){$e_2$}

		\draw[solid](0,9)--(2,9);	
		\draw[solid](2,7)--(2,9);
		\draw[solid](0,9)--(0,7);
		\draw[solid](0,7)--(2,7);
		\draw[solid](0,9)--(2,7);
		\draw[solid](2,9)--(0,7);

		\end{tikzpicture}
	\end{center}

\vspace{0.2cm}

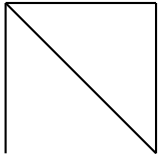
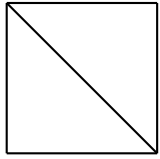
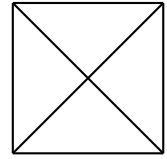
\hspace{0.1cm}(4) chordal graph with \hspace{1.0cm}(5) chordal graph  with \hspace{1.0cm}(6) complete graph  with

\hspace{0.5cm} $e_i \geq 2$ for $i=1,2$ \hspace{1.5cm} $e_i \geq 2$ for $i=1,2$ \hspace{1.9cm}$e_i \geq 2$ for $i=1,2$

\vspace{0.7cm}	

Claim:   In each of the six cases above, there exists some  $e\in E(H_\omega)$ such that $\omega(e)=1$.

If  $H_\omega$ is an induced subgraph of $G_\omega$, as shown in  case (1)  or case (6), and $\omega(e)\ge 2$ for all $e\in E(H_\omega)$, then $G_\omega$ has an induced subgraph $C_\omega^3$,  such that all of whose edges have non-trivial weights, which contradicts the hypothesis. Hence there exists some  $e\in E(H_\omega)$ such that $\omega(e)=1$.

If  $H_\omega$ is an induced subgraph of $G_\omega$, as shown in one of the  cases (2)-(5), and $\omega(e)\ge 2$ for all $e\in E(H_\omega)$, then $G_\omega$ has an induced path $P_\omega^3$,  such that all of its edges have non-trivial weights, a contradiction.

Without loss of generality, we can assume that $\omega(\{u_1,u_2\})=1$. Note that   $(a_1,\ldots,a_n)$  satisfies the  expression (\ref{eqn:SES-1}), we have  $a_{u_1}\ge \lambda_1\omega_1>0$ and $a_{u_2}\ge \lambda_2\omega_2>0$, where $a_{u_1}$ and $a_{u_2}$ are the $u_1$-th and $u_2$-th entries of  $(a_1,\ldots,a_n)$ respectively. It follows that $\lceil a_{u_1}\rceil\ge 1$ and $\lceil a_{u_2}\rceil\ge 1$, so $x_{u_1}x_{u_2}$  divides $\xb^{\ab }$. So $\xb^{\ab}\in I(G_\omega)$. This completes the proof.
\end{proof}

\section{Normality of edge ideals of some edge-weighted graphs   }
In this section, we will show that for a weighted star graph or a weighted path or a weighted cycle with  the  edge ideal  $I$, if  $I$ is integrally closed, then $I$ is normal.
First,  we   recall a key notion from \cite{HT}, which will be helpful in understanding the
integral closure of ideals.

Let $S=\KK[x_{1},\dots, x_{n}]$  be a  polynomial ring in $n$ variables over a field $\KK$ and $\xb^{\ab}=x_1^{a_1}\cdots x_n^{a_n}\in S$ be  a monomial with an exponent vector $\ab=(a_1,\ldots,a_n)$. Let $I\subset S$ be a monomial ideal with $\mathcal{G}(I)=\{\xb^{{\bb}_1},\ldots, \xb^{{\bb}_m}\}$.  We call the $n\times m$ matrix $M$, whose columns are exponent vectors ${\bb}_1,\ldots,{\bb}_m$, the \emph{exponent matrix} of $I$. We set
\begin{align*}
v_\mathbf{a}(I)&=\max\,\{\mathbf{1}^{m} \cdot \mathbf{y} \mid  M \cdot \mathbf{y} \leq \mathbf{a}\ \text{with}\  \mathbf{y} \in \mathbb{Z}_{+}^{m}\},\\
v_\mathbf{a}^\ast(I)&=\max\,\{\mathbf{1}^{m} \cdot \mathbf{y} \mid  M \cdot \mathbf{y} \leq \mathbf{a}\ \text{with}\  \mathbf{y} \in \mathbb{R}_{\geq 0}^{m}\},
\end{align*}
where $\mathbb{R}_{\geq 0}$ is the set of all non-negative real numbers.

\begin{Lemma}{\em (\cite[ Proposition 3.1]{T})}
	\label{normal}
Let $I\subset S$ be a monomial ideal.  Then
 \begin{enumerate}
\item  \label{normal-1}$\xb^{\ab} \in I^k$ if and only if $v_\mathbf{a}(I)\geq k$,
\item  \label{normal-2}$\xb^{\ab} \in \overline{I^k}$ if and only if $v_\mathbf{a}^\ast(I)\geq k$.
 \end{enumerate}
\end{Lemma}

\begin{Remark}\label{Rounding}
	Let $k$ be  a positive integer.
\begin{enumerate}
\item \label{Rounding_1}If  $x+y \geq k$ with $x,y \in \mathbb{R}_{\geq 0}$, then  $\lceil x \rceil+\lfloor y \rfloor \geq k$; 	
\item \label{Rounding_2} If $x+y \leq k$ with $x,y \in \mathbb{R}_{\geq 0}$, then $\lceil x \rceil+\lfloor y \rfloor \leq k$.
 \end{enumerate}
  where $\lfloor y \rfloor$ is the  largest integer $\le y$.	
  \end{Remark}

\begin{proof}
(1) If  $x+y \geq k$, then $\lceil x \rceil+y \geq k$, i.e, $y\geq k-\lceil x \rceil$. Since  $k-\lceil x \rceil$ is an integer, we have  $\lfloor y \rfloor \geq k-\lceil x \rceil$,
i.e, $\lceil x \rceil+\lfloor y \rfloor \geq k$.

(2) If $x+y \leq k$, then $x+\lfloor y \rfloor \leq k$, i.e, $x \leq k-\lfloor y \rfloor$. It follows that $\lceil x \rceil \leq k-\lfloor y \rfloor$,  i.e, $\lceil x \rceil+\lfloor y \rfloor \leq k$.
\end{proof}

We now prove some of the main results of this section.
\begin{Theorem}
Let $G_\omega$ be a weighted star graph with $n$ vertices, and let $I=I(G_\omega)$ be its  edge ideal. If $I$ is integrally closed, then $I$ is normal.		
\end{Theorem}

\begin{proof}Let $E(G_\omega)=\{e_1,\ldots,e_{n-1}\}$, where  $e_i=\{i,n\}$ and $\omega_i=\omega(e_i)$ for $i\in [n-1]$. Since $I$ is integrally closed, $G_\omega$ has at most one edge with non-trivial weight by Theorem \ref{main}. If $G_\omega$ is trivially weighted, then $I$ is normal by
\cite[Proposition 2.1 and Corollary 2.8]{AVV} and \cite[Proposition 2.1.2]{HSV}. Now we assume that $G_\omega$ has an edge with non-trivial weight.
In this case,  we can assume  by symmetry that $\omega_1\ge 2$ and $\omega_i=1$ for $i\in [n-1]$ with $i\ne 1$. We will prove that $\overline{I^k}=I^k$ for all $k\ge 2$.
Since $I^k\subseteq \overline{I^k}$ is always valid,  it suffices  to prove that  $\overline{I^k}\subseteq I^k$.

Let $\xb^{\ab}=x_1^{a_1}\cdots x_n^{a_n} \in \mathcal{G}(\overline{I^k})$, then $v_\mathbf{a}^\ast(I)\geq k$ by Lemma \ref{normal}(\ref{normal-2}). From the definition of $v_\mathbf{a}^\ast(I)$  it follows that there exists the
vector $\mathbf{y}=(y_1,y_2,\dots,y_{n-1})^T\in \mathbb{R}_{\geq 0}^{n-1}$ which  satisfies the following system of inequalities
\[
\begin{array}{cccc}
(1)\quad&
\left\{
\begin{aligned}
y_1+\cdots+y_{n-1} &\geq k, &\circled{1}\\
\omega_1y_1 &\leq  a_1, &\circled{2}\\
y_2 &\leq a_2,\\
&\vdots\\
y_{n-1} &\leq a_{n-1},\\
\omega_1y_1+y_2+\cdots+y_{n-1} &\leq a_n.&\circled{3}
\end{aligned}
\right.
\end{array}
\]
We distinguish between the following two cases:
	\begin{enumerate}
\item[(1)] If there exists some  $i\in [n-1]\backslash \{1\}$  such that $a_i \geq k$, then $\xb^{\ab}$ can be divisible by $(x_ix_n)^k$ from $\circled{1}$ and $\circled{3}$ in system (1),  which implies that $\xb^{\ab} \in I^k$.
\item[(2)] If  $a_i<k$ for all $i\in [n-1]\backslash \{1\}$. We consider the following three subcases:
\begin{enumerate}
\item[(i)] If there exists some $j\in [n-2]\backslash \{1\}$ such that
     $a_2+\cdots+a_{j+1}=k$, then in this case  $\xb^{\ab}$ can be divisible  by
    $(x_2x_n)^{a_2}(x_3x_n)^{a_3} \cdots (x_jx_n)^{a_j}(x_{j+1}x_n)^{b_{j+1}}$ where $b_{j+1}=k-(a_2+\cdots+a_j)$, so that $\xb^{\ab} \in I^k$.
\item[(ii)] If there exists some $j\in [n-2]\backslash \{1\}$ such that
     $a_2+\cdots+a_{j+1}>k$, then we choose the  maximum  $\ell$ such that $a_2+\cdots+a_{\ell}\le k$.
     In this case, $\xb^{\ab}$ can be divisible  by
    $(x_2x_n)^{a_2}(x_3x_n)^{a_3} \cdots (x_{\ell}x_n)^{a_{\ell}}(x_{\ell+1}x_n)^{b_{\ell+1}}$ where $b_{\ell+1}=k-(a_2+\cdots+a_{\ell})$, so  $\xb^{\ab} \in I^k$.
\item[(iii)] If $a_2+\cdots+a_{n-1}<k$.  In this case, let $b=a_2+\cdots+a_{n-1}$, then  $y_2+\cdots+y_{n-1} \le b<k$. It follows from $\circled{1}$ in system (1) that
		\begin{equation}
		y_1 \geq k-b.\label{eqn: inequality1}
		\end{equation}
Therefore  $a_1 \geq \omega_1y_1 \geq \omega_1(k-b)$ by $\circled{2}$ in system (1).
		By $\circled{3}$ in system (1) and the inequality (\ref{eqn: inequality1}), we get
		\begin{align*}
		a_n & \geq\omega_1y_1+y_2+\cdots+y_{n-1}\\
		& = y_1+\dots+y_{n-1}+(\omega_1-1)y_1\\
		& \geq k+(\omega_1-1)(k-b)\\
    &=b+\omega_1(k-b).
		\end{align*}
It  follows that $\xb^{\ab}$ is divisible  by $(x_2x_n)^{a_2}\cdots(x_{n-1}x_{n})^{a_{n-1}}(x_1^{\omega_1}x_n^{\omega_1})^{k-b}$, so  $\xb^{\ab} \in I^k$, since $(k-b)+a_2+\cdots+a_{n-1}=(k-b)+b=k$.\qedhere
\end{enumerate}
 \end{enumerate}
\end{proof}

\begin{Lemma}
	\label{divide} Let $n\ge 2$ be an integer and let
  $\xb^{\ab}=x_1^{a_1}\cdots x_n^{a_n}\in S$ be a monomial whose exponent vector  $\ab=(a_1,\ldots,a_n)$  satisfies one of the following two conditions:
    \begin{enumerate}
\item  \label{divide-1} $n\ge 3$, $a_j\geq a_{j-1}-a_{j-2}+\cdots+(-1)^{i-1}a_{j-i}+\cdots+(-1)^{j-2}a_{1}$ for each $j=2,\ldots,n-1$ and
$a_n \leq a_{n-1}-a_{n-2}+\cdots+(-1)^{i-1}a_{n-i}+\cdots+(-1)^{n-2}a_{1}$.
\item  \label{divide-2}  $n=2r$ and $a_{2i-1} \geq a_{2i}$ for each $i=1,\dots,r$.
 \end{enumerate}
Suppose that a vector $\mathbf{y}=(y_1,y_2,\dots,y_n)^T\in \mathbb{R}_{\geq 0}^{n}$  satisfies the following inequality system
  	\[
\begin{array}{cc}
(2)\quad&
	\left\{
	\begin{aligned}
	y_1  &\leq  a_1,\\
	y_1+y_2 &\leq a_2,\\
	y_2+y_3 &\leq a_3,\\
	&\vdots \\
	y_{n-2}+y_{n-1} &\leq a_{n-1},\\
	y_{n-1}+y_n &\leq a_n.
	\end{aligned}
\right.
\end{array}
\]
 Let $h=\lceil y_1+\dots+y_{n}\rceil$, then there exist at least $h$ monomials $e_1,\ldots,e_h\in \{x_1x_2,x_2x_3,\\
\dots,x_{n-1}x_n\}$ such that $\xb^{\ab}$ can be divisible by $\prod_{i=1}^{h}e_i$.
\end{Lemma}
\begin{proof}
(1) Let $b_{j-1}=a_{j-1}-a_{j-2}+\cdots+(-1)^{i-1}a_{j-i}+\cdots+(-1)^{j-2}a_{1}$ for  $j=2,\dots,n-1$, then by the assumption we have $a_j\ge b_{j-1}$ for  $j=2,\dots,n-1$, and $a_{n}\le b_{n-1}$. Meanwhile, we also get that
$b_1=a_1$,  $b_{j}+b_{j-1}=a_{j}$ for  $j=2,\dots,n-1$, and $a_{n}\le a_{n-1}-b_{n-2}$. It follows that $a_{n-1}\ge a_{n}+b_{n-2}$, and $b_j\ge 0$  since $a_j\ge b_{j-1}$.  By comparing the indices of each variable we find  that $\xb^{\ab}$ can be divisible by $(x_1x_2)^{b_1}(x_2x_3)^{b_2}\cdots (x_{n-2}x_{n-1})^{b_{n-2}}(x_{n-1}x_n)^{a_n}$.

From the  inequality system (2) above, we see that if $n=2r$ then  $b_1+\cdots+b_{n-2}+a_n=\sum \limits_{i=1}^{r}a_{2i}\ge y_1+\dots+y_n$;  if $n=2r-1$ then $b_1+\cdots+b_{n-2}+a_n=\sum \limits_{i=1}^{r}a_{2i-1}\ge y_1+\dots+y_n$. In both cases, we always have $b_1+\cdots+b_{n-2}+a_n\ge h$,
as desired.

(2) If $n=2r$ and $a_{2i} \leq a_{2i-1}$ for each $i=1,\dots,r$, then it is clear that $\xb^{\ab}$ can be  divisible by $\prod\limits_{i=1}^{r}(x_{2i-1}x_{2i})^{a_{2i}}$ and $\sum \limits_{s=1}^{r}{a_{2s}} \geq y_1+\dots+y_{n}$ by the system (2), thus $\sum \limits_{s=1}^{r}{a_{2s}}\ge h$.
As expected.
\end{proof}

 Applying  similar techniques, we can get the following lemma.
\begin{Lemma}
	\label{divide2}
Let $n\ge 2$ be an integer and let $\xb^{\ab}=x_1^{a_1}\cdots x_n^{a_n}\in S$ be a monomial whose exponent vector  $\ab=(a_1,\ldots,a_n)$  satisfies one of the following four conditions:
    \begin{enumerate}
	\item  \label{divide-3} $a_j\geq a_{j-1}-a_{j-2}+\cdots+(-1)^{i-1}a_{j-i}+\cdots+(-1)^{j-2}a_{1}$ for each $j=2,\ldots,n$.
	\item  \label{divide-4} $n\ge 3$, $a_j\geq a_{j-1}-a_{j-2}+\cdots+(-1)^{i-1}a_{j-i}+\cdots+(-1)^{j-2}a_{1}$ for each  $j=2,\ldots,n-1$ and
	$a_n \leq a_{n-1}-a_{n-2}+\cdots+(-1)^{i-1}a_{n-i}+\cdots+(-1)^{n-2}a_{1}$.
	\item  \label{divide-5}  $n=2r$  and $a_{2i-1} \geq a_{2i}$ for each  $i\in [r]$.
	\item  \label{divide-6}   $n=2r-1$ with $r\geq 2$ and $a_{2i-1} \geq a_{2i}$ for each $i\in [r-1]$.
\end{enumerate}
	Suppose that a vector $\mathbf{y}=(y_1,y_2,\dots,y_{n-1})^T\in \mathbb{R}_{\geq 0}^{n-1}$  satisfies the following inequality system
	\[
	\begin{array}{cc}
	(3)\quad&
	\left\{
	\begin{aligned}
	y_1  &\leq  a_1,\\
	y_1+y_2 &\leq a_2,\\
	y_2+y_3 &\leq a_3,\\
	&\vdots \\
	y_{n-2}+y_{n-1} &\leq a_{n-1},\\
	y_{n-1} &\leq a_n.
	\end{aligned}
	\right.
	\end{array}
	\]
	Let $h=\lceil y_1+\dots+y_{n-1}\rceil$, then there exist at least $h$ monomials $e_1,\ldots,e_h\in \{x_1x_2,x_2x_3,
	\ldots,x_{n-1}x_n\}$ such that $\xb^{\ab}$ can be  divisible by $\prod_{i=1}^{h}e_i$.	
\end{Lemma}

\begin{proof} (1) For each $j=2,\dots,n$, let $b_{j-1}=a_{j-1}-a_{j-2}+\cdots+(-1)^{i-1}a_{j-i}+\cdots+(-1)^{j-2}a_{1}$, then $b_1=a_1$ and  $b_{j}+b_{j-1}=a_{j}$ for all $j=2,\ldots, n$.
  It follows that $b_j \geq 0$ from the assumption $a_j\ge b_{j-1}$. By comparing the indices of each variable we see that $\xb^{\ab}$ can be  divisible by $(x_1x_2)^{b_1}(x_2x_3)^{b_2}\cdots (x_{n-2}x_{n-1})^{b_{n-2}}(x_{n-1}x_n)^{b_{n-1}}$.  Note that
  $\sum\limits_{i=1}\limits^{n-1}{b_i}=b_1+(b_2+b_3)+\cdots (b_{n-2}+b_{n-1})=\sum\limits_{i=1}\limits^{r}{a_{2i-1}} \geq y_1+\cdots+y_{n-1}$ when  $n=2r$,
  $\sum\limits_{i=1}\limits^{n-1}{b_i}=(b_1+b_2)+\cdots (b_{n-2}+b_{n-1})=\sum\limits_{i=1}\limits^{r-1}{a_{2i}} \geq y_1+\cdots+y_{n-1}$ when $n=2r-1$.
  In both cases, we always have $\sum\limits_{i=1}\limits^{n-1}{b_i} \geq h$, as desired.

  (2)  and (3) can be shown by arguments similar to Lemma	\ref{divide}.

  (4)  If $n=2r-1$ with $r\geq 2$ and  $a_{2i-1} \geq a_{2i}$ for each $i\in [r-1]$, then it is clear that $\xb^{\ab}$ is divisible by $(x_1x_2)^{a_2}(x_3x_4)^{a_4}(x_5x_6)^{a_6}\cdots(x_{n-4}x_{n-3})^{a_{n-3}}(x_{n-2}x_{n-1})^{a_{n-1}}$. And $\sum \limits_{i=1}^{r-1}{a_{2i}} \geq y_1+\dots+y_{n-1}$ by the system (3), which implies that $\sum \limits_{i=1}^{r-1}{a_{2i}}\ge h$, as wished.
\end{proof}

\begin{Theorem}
	\label{divide3}Let $n\ge 2$ be an integer and let
  $\xb^{\ab}=x_1^{a_1}\cdots x_n^{a_n}\in S$ be a monomial with exponent vector   $\ab=(a_1,\ldots,a_n)$. Suppose that a vector $\mathbf{y}=(y_1,y_2,\dots,y_{n-1})^T\in \mathbb{R}_{\geq 0}^{n-1}$  satisfies the following system of inequalities
	\[
	\begin{array}{cc}
(4)\quad &
\left\{
	\begin{aligned}
	y_1  &\leq  a_1,\\
	y_1+y_2 &\leq a_2,\\
	y_2+y_3 &\leq a_3,\\
	\vdots \\
	y_{n-2}+y_{n-1} &\leq a_{n-1},\\
	y_{n-1} &\leq a_n.
	\end{aligned}
	\right.
	\end{array}
\]
Let $h=\lceil y_1+\cdots+y_{n-1} \rceil$, then there exist at least $h$ monomials  $e_1,\ldots,e_h\in \{x_1x_2,x_2x_3,\dots,x_{n-1}x_n\}$ such that $\xb^{\ab}$ is divisible by $\prod_{i=1}^{h}e_i$.
\end{Theorem}
\begin{proof}If  $n=2$, then it is trivial. If $n=3$, then by comparing the sizes of $a_1$, $a_2$ and $a_3$, we see that $\ab=(a_1,a_2,a_3)$ satisfies  Lemma \ref{divide2}, so the desired result follows from Lemma \ref{divide2}. Now we  assume that  $n\ge 4$. If $\ab=(a_1,\ldots,a_n)$  satisfies   Lemma \ref{divide2}, then the desired result follows from Lemma \ref{divide2}.
Otherwise, there are two subcases:
\begin{enumerate}	
	\item[(i)] When $a_1>a_2$. If $n=2r$, then there exists some $t\in [r-1]$ such that  $a_{2i-1} \ge a_{2i}$ for each $i\in [t]$ and  $a_{2t+1} < a_{2t+2}$. In this case,  the vector $(a_1,\ldots,a_{2t})$ satisfies the assumption (2) of Lemma 4.4.  Otherwise, if $n=2r-1$, then  there exists some  $t'\in [r-2]$ such that $a_{2i-1} \ge a_{2i}$ for each $i\in [t']$ and $a_{2t'+1} < a_{2t'+2}$. In this case, the vector $(a_1,\ldots,a_{2t'})$ satisfies the assumption (2) of Lemma 4.4.
	\item[(ii)]	
If $a_1 \leq a_2$, then there exists some $s\in [n-1]$ such that $a_j\geq a_{j-1}-a_{j-2}+\cdots+(-1)^{i-1}a_{j-i}+\cdots+(-1)^{j-2}a_{1}$ for each $j=2,\ldots,s-1$ and $a_s \leq a_{s-1}-a_{s-2}+\cdots+(-1)^{i-1}a_{s-i}+\cdots+(-1)^{s-2}a_{1}$.	In this case, the vector $(a_1,\ldots,a_{s})$ satisfies the assumption (1) of Lemma 4.4.
\end{enumerate}
When $a_1>a_2$. If  $n=2r$, then  we let $s=2t$. Otherwise, if $n=2r-1$, then  we set $s=2t'$. Thus the first $s$ components of the vector $(a_1,\ldots,a_s,0,\ldots,0)$ satisfy the following
 system (5) of inequalities
\[
 \begin{array}{cccc}
 (5)\quad&
 \left\{
 \begin{aligned}
 y_1  &\leq  a_1,\\
 y_1+y_2 &\leq a_2,\\
 y_2+y_3 &\leq a_3,\\
 &\vdots \\
 y_{s-2}+y_{s-1} &\leq a_{s-1},\\
 y_{s-1}+y_s &\leq a_s.
 \end{aligned}
 \right.&\quad\quad(6)\quad&
 \left\{
 \begin{aligned}
 y_{s+1}  &\leq  a_{s+1},\\
 y_{s+1}+y_{s+2} &\leq a_{s+2},\\
 y_{s+2}+y_{s+3} &\leq a_{s+3},\\
 \vdots \\
 y_{n-2}+y_{n-1} &\leq a_{n-1},\\
 y_{n-1} &\leq a_{n}.
 \end{aligned}
 \right.
 \end{array}
 \]
Let $\mathbf{c}=(a_1,\ldots,a_s)$, then the vector $\mathbf{c}$ satisfies  Lemma \ref{divide}, so there exist at least $h'$ monomials  $e_1,\ldots,e_{h'}\in \{x_1x_2,x_2x_3,\dots,x_{s-1}x_s\}$ such that $\xb^{\cb}$ can be divisible by $\prod_{i=1}^{h'}e_i$, where $h'=\lceil y_1+\cdots+y_s \rceil$.
Now we consider two subcases depending on whether $s+1=n$ or not:
(a) If $s+1=n$, then in this case we choose $h=h'$ and the result follows. (b) If $s+1<n$, then   the last $(n-s)$ components of the vector $(0,\ldots,0, a_{s+1},\ldots,a_n)$ satisfy the above  system (6). In this case, let $\mathbf{d}=(a_{s+1},\ldots,a_n)$, then the vector
$\mathbf{d}$ satisfies   Lemma 	\ref{divide2}. It follows from  Lemma 	\ref{divide2} that  there exist at least $h''$ monomials  $f_1,\ldots,f_{h''}\in \{x_{s+1}x_{s+2},x_{s+2}x_{s+3},\dots,x_{n-1}x_n\}$ such that $\xb^{\db}$ can be divisible by $\prod_{i=1}^{h''}f_i$, where $h''=\lceil y_{s+1}+\cdots+y_{n-1} \rceil$.
Therefore  $\xb^{\ab}$ can be  divide by $(\prod_{i=1}^{h'}e_i)(\prod_{j=1}^{h^{''}}{f_j})$ and  $h'+h'' \ge h$.
Otherwise, by repeating the above discussion, we can decompose the set  $\{a_{s+1},a_{s+2},\dots,a_n\}$ into disjoint  unions of  finite  continuous segments, say $t$, such that for  each $i\in [t-1]$, the $i$-th continuous segment satisfies the  assumption (1) or (2) in  Lemma \ref{divide} and the $t$-th segment satisfies one of the four conditions in Lemma \ref{divide2}. Thus  we can write $\{a_{s+1},\ldots,a_n\}$ as $\{a_{s+1},\ldots,a_n\}=\bigsqcup\limits_{i=1}^{t}C_i$, where $C_i=\{a_{p_i+1},a_{p_i+2},\ldots,a_{p_{i+1}}\}$ for each $i\in [t]$ with $p_1=s$  and  $p_{t+1}=n$. Note that for each $i\in [t-1]$,   $C_i$  satisfies the following system  (7) of inequalities
 and
$C_t$  satisfies the following  system  (8) of inequalities when $|C_t|\ge 2$.
\[
	\begin{array}{cccc}
	(7)\quad&
	\left\{
	\begin{aligned}
	y_{p_i+1}  &\leq  a_{p_i+1},\\
	y_{p_i+1}+y_{p_i+2} &\leq a_{p_i+2},\\
	y_{p_i+2}+y_{p_i+3} &\leq a_{p_i+3},\\
	\vdots \\
	y_{p_{i+1}-2}+y_{p_{i+1}-1} &\leq a_{p_{i+1}-1},\\
	y_{p_{i+1}-1}+y_{p_{i+1}} &\leq a_{p_{i+1}}.
	\end{aligned}
	\right.&(8)\quad&
	\left\{
	\begin{aligned}
	y_{p_t+1} &\leq  a_{p_t+1},\\
	y_{p_t+1}+y_{p_t+2} &\leq a_{p_t+2},\\
	y_{p_t+2}+y_{p_t+3} &\leq a_{p_t+3},\\
	\vdots \\
	y_{p_{t+1}-2}+y_{p_{t+1}-1} &\leq a_{p_{t+1}-1},\\
	y_{p_{t+1}-1} &\leq a_{p_{t+1}}.
	\end{aligned}
	\right.
	\end{array}
	\]
It follows from Lemma \ref{divide} that for each $i\in [t-1]$, there exist at least $h_i$ monomials $u_{i1},\ldots,u_{ih_i} \in \{x_{p_i+1}x_{p_i+2},x_{p_i+2}x_{p_i+3},\ldots,x_{p_{i+1}-1}x_{p_{i+1}}\}$ such that $\xb^{\c_ib}$ can be  divisible by $\prod_{j=1}^{h_i}u_{ij}$, where the exponent vector $\mathbf{c_i}=(0,\ldots,0,a_{p_i+1},a_{p_i+2},\ldots,a_{p_{i+1}},0,\ldots,0)$  and each $h_i=\lceil y_{p_i+1}+\cdots+y_{p_{i+1}} \rceil$.
	If $|C_t|=1$, then  in this case  we have $p_t+1=n$, which implies that  $h'+\sum\limits_{i=1}\limits^{t-1}{h_i} \ge \lceil y_1+\cdots+y_s \rceil+\sum\limits_{i=1}\limits^{t-1}\lceil y_{p_i+1}+\cdots+y_{p_{i+1}} \rceil\ge h$ and
$\xb^{\ab}$ can be  divisible by $(\prod_{i=1}^{h'}e_i)(\prod_{i=1}^{t-1}\prod_{j=1}^{h_i}u_{ij})$, since $\xb^{\ab}=\xb^{\cb}(\prod_{i=1}^{t-1}{\xb^{\cb_i}})x_n^{a_n}$. If $|C_t|\ge 2$, then by Lemma \ref{divide2},
 there exist at least $h_t$ monomials  $u_{t1},\ldots,u_{th_t} \in \{x_{p_t+1}x_{p_t+2},x_{p_t+2}x_{p_t+3},\ldots,x_{n-1}x_{n}\}$ such that $\xb^\mathbf{c_t}$ can be  divisible by
$\prod_{j=1}^{h_t}u_{tj}$, where $h_t=\lceil y_{p_t+1}+\cdots+y_{n-1} \rceil$  and the exponent vector $\mathbf{c_t}=(0,\ldots,0,a_{p_t+1},a_{p_t+2},\ldots,a_{n})$.
In this case, $h'+\sum\limits_{i=1}\limits^{t}{h_i} \ge \lceil y_1+\cdots+y_s \rceil+\sum\limits_{i=1}\limits^{t}\lceil y_{p_i+1}+\cdots+y_{p_{i+1}} \rceil\ge h$
and  $\xb^{\ab}$ can be  divided by $(\prod_{i=1}^{h'}e_i) (\prod_{i=1}^{t}\prod_{j=1}^{h_i}u_{ij})$. This  completes the proof.
\end{proof}

\begin{Theorem}
\label{cycle1}
Let $C_\omega^n$ be a weighted cycle on the set $[n]$, where exactly three edges have non-trivial weights. Let $I=I(C_\omega^n)$ be the  edge ideal of the  cycle  $C_\omega^n$. If $I$ is integrally closed, then $I$ is normal.	
\end{Theorem}
\begin{proof} If $C_\omega^n$ has exactly three edges with non-trivial weights, then, by  Theorem
\ref{main} we have $n=6$ and any two edges in these non-trivially weighted edges do not share a common vertex.
 By symmetry, let $E(C_\omega^n)=\{e_1,\dots,e_6\}$ with each $\omega_i=\omega(e_i)$, where $e_i=\{i,i+1\}$ for $i \in [5]$ and $e_6=\{6,1\}$. Then we can assume that $\omega_1, \omega_3, \omega_5 \geq 2$ and  $\omega_2=\omega_4=\omega_6=1$.\\
We will prove that $\overline{I^k}=I^k$ for all $k\ge 2$.
Since $I^k\subseteq \overline{I^k}$ is always valid,  it suffices to prove that  $\overline{I^k}\subseteq I^k$.\\
Let $\xb^{\ab}=x_1^{a_1}\cdots x_6^{a_6} \in \mathcal{G}(\overline{I^k})$, then $v_\mathbf{a}^\ast(I)\geq k$ by Lemma \ref{normal}(\ref{normal-2}).  It follows from the definition of $v_\mathbf{a}^\ast(I)$ that there exists the
vector $\mathbf{y}=(y_1,y_2,\dots,y_6)^T\in \mathbb{R}_{\geq 0}^{6}$ satisfying the following system  (9) of inequalities
\[
\begin{array}{cccc}
(9)\quad&
\left\{
\begin{aligned}
y_1+\dots+y_6 &\geq k,&\circled{1}\\
\omega_1y_1+y_6 &\leq a_1,&\circled{2}\\
\omega_1y_1+y_2 &\leq a_2,&\circled{3}\\
y_2+\omega_3y_3 &\leq a_3,&\circled{4}\\
\omega_3y_3+y_4 &\leq a_4,&\circled{5}\\
y_4+\omega_5y_5 &\leq a_5,&\circled{6}\\
\omega_5y_5+y_6 &\leq a_6.&\circled{7}
\end{aligned}
\right.
\end{array}
\]
In this case, we get  $a_2+a_3 \geq \omega_1y_1+2y_2+\omega_3y_3 \geq 2(y_1+y_2+y_3)$  by $\circled{3}$ and  $\circled{4}$  in  system (9). Similarly, $a_4+a_5 \geq  2(y_3+y_4+y_5)$  and $a_1+a_6 \geq 2(y_1+y_5+y_6)$.  Due to symmetry, we
only need to prove that $\xb^{\ab} \in I^k$ provided that  $a_2+a_3 \geq 2(y_1+y_2+y_3)$. We distinguish between the following two cases:

(I)  If $a_2,a_3 \geq y_1+y_2+y_3$, then $a_2,a_3 \geq \lceil y_1+y_2+y_3\rceil $. By  the   system (9), we have    $a_4\ge \lceil y_4 \rceil$, $a_5\ge \omega_5\lfloor y_5\rfloor+ \lceil y_4\rceil$,  $a_6\ge \omega_5\lfloor y_5\rfloor+\lceil y_6\rceil$ and  $a_1\ge \lceil y_6\rceil$,  so  $\xb^{\ab}$ is divisible by
$(x_2x_3)^{\lceil y_1+y_2+y_3 \rceil}(x_4x_5)^{\lceil y_4 \rceil}(x_5^{\omega_5}x_6^{\omega_5})^{\lfloor y_5 \rfloor}(x_6x_1)^{\lceil y_6 \rceil}$.
Note that $\lceil y_1+y_2+y_3 \rceil+\lceil y_4 \rceil+y_5+\lceil y_6 \rceil \geq k$, since $y_1+\dots+y_6 \geq k$. It follows that $y_5 \geq k-(\lceil y_1+y_2+y_3 \rceil+\lceil y_4 \rceil+\lceil y_6 \rceil )$, which forces that $\lfloor y_5\rfloor \geq k-(\lceil y_1+y_2+y_3 \rceil+\lceil y_4 \rceil+\lceil y_6 \rceil )$. since $k-(\lceil y_1+y_2+y_3 \rceil+\lceil y_4 \rceil+\lceil y_6 \rceil )$ is an integer.    Therefore,  $\lceil y_1+y_2+y_3 \rceil+\lceil y_4 \rceil+\lfloor y_5 \rfloor+\lceil y_6 \rceil \geq k$, which gives the desired result.

(II)  If $a_2 \geq y_1+y_2+y_3 >a_3$, or $a_3 \geq y_1+y_2+y_3>a_2$. By  symmetry,  it suffices to prove that $\xb^{\ab} \in I^k$ provided that  $a_2 \geq y_1+y_2+y_3 >a_3$.
In this case, we get $y_1 =y_1+y_2+y_3 -(y_2+y_3) \geq (y_1+y_2+y_3)-a_3$  by $\circled{4}$ in  system (9). It follows from $\circled{2}$ in  system (9) that
$ a_1 \geq \omega_1y_1+y_6 \geq \omega_1(y_1+y_2+y_3-a_3)+y_6$.
So by Remark \ref{Rounding}(\ref{Rounding_2}), we get
\begin{equation}
a_1 \geq \omega_1\lfloor y_1+y_2+y_3-a_3 \rfloor+\lceil y_6 \rceil. \label{eqn: May3_1}
\end{equation}
On the other hand, from  $\circled{3}$ and $\circled{4}$ in  system (9), we have
\begin{align*}
a_2+a_3  &\ge\omega_1y_1+2y_2+\omega_3y_3\\
&=\omega_1(y_1+y_2+y_3-a_3)+\omega_1(a_3-y_2-y_3)+2y_2+\omega_3y_3\\
& \ge \omega_1(y_1+y_2+y_3-a_3)+2(a_3-y_2-y_3)+2y_2+2y_3\\
& \ge \omega_1(y_1+y_2+y_3-a_3)+2a_3,
\end{align*}
which forces that  $a_2 \geq \omega_1(y_1+y_2+y_3-a_3)+a_3$, so we have
\begin{equation}
a_2 \geq \omega_1\lfloor y_1+y_2+y_3-a_3\rfloor+a_3.
\end{equation}
We consider the following two subcases:

(a) If  $a_4\ge a_5$, then $\xb^{\ab}$ is divisible by  $(x_1x_2)^{\omega_1\lfloor y_1+y_2+y_3-a_3 \rfloor} (x_2x_3)^{a_3}(x_4x_5)^{a_5}(x_6x_1)^{\lceil y_6 \rceil}$  by  system (9).
By $\circled{6}$ in  system (9), we have   $a_5+\lceil y_6 \rceil \geq y_4+y_5+y_6$. Thus $(y_1+y_2+y_3-a_3)+a_3+a_5+\lceil y_6 \rceil \geq y_1+y_2+y_3+y_4+y_5+y_6\ge  k$, i.e., $y_1+y_2+y_3-a_3 \geq k-(a_3+a_5+\lceil y_6 \rceil)$. This yields that $\lfloor y_1+y_2+y_3-a_3 \rfloor \geq k-(a_3+a_5+\lceil y_6 \rceil)$, i.e., $\lfloor y_1+y_2+y_3-a_3 \rfloor + a_3+a_5+\lceil y_6 \rceil \geq k$, so $\xb^{\ab} \in I^k$.

(b) If $a_4<a_5$. We will prove that  $\xb^{\ab} \in I^k$ in the following two scenarios:

(i) If $a_6\geq y_5+y_6+y_1>a_1$, then  by similar arguments as for the condition that  $a_2 \geq y_1+y_2+y_3 >a_3$ and $a_4\ge a_5$, it follows  that  $\xb^{\ab}$ is divisible by  $(x_5x_6)^{\omega_5\lfloor y_1+y_5+y_6-a_1\rfloor}(x_1x_6)^{a_1}\cdot(x_2x_3)^{a_3}(x_4x_5)^{\lceil y_4 \rceil}$ and $\lfloor y_1+y_5+y_6-a_1\rfloor+a_1+a_3+\lceil y_4 \rceil\ge k$.

(ii) If $a_1\geq y_5+y_6+y_1>a_6$, then again by similar arguments as for the condition that  $a_2 \geq y_1+y_2+y_3 >a_3$ and $a_4\ge a_5$, we get that  $\xb^{\ab}$ is divisible by  $(x_1x_2)^{\omega_1\lfloor y_1+y_5+y_6-a_6\rfloor}(x_1x_6)^{a_6}(x_4x_5)^{a_4}(x_2x_3)^{\lceil y_2 \rceil}$ and $\lfloor y_1+y_5+y_6-a_6\rfloor+a_6+a_4+\lceil y_2 \rceil\ge k$.
In both cases, we always have   $\xb^{\ab}\in I^k$.
\end{proof}

\begin{Theorem}
\label{cycle2}
Let $C_\omega^n$ be a weighted cycle on the vertex set $[n]$, where exactly two edges have non-trivial weights. Let $I=I(C_\omega^n)$ be the  edge ideal of the cycle  $C_\omega^n$. If $I$ is integrally closed, then $I$ is normal.	
\end{Theorem}
\begin{proof}Let $E(C_\omega^n)=\{e_1,\dots,e_n\}$ with  each $\omega_i=\omega(e_i)$ and $e_i=\{i,i+1\}$, where $i$ is identified with the integer $0<j\le n$ such that $j\equiv i\pmod n$. Since $C_\omega^n$ has exactly two edges with non-trivial weights,   we can assume by symmetry that $\omega_1, \omega_3 \geq 2$ and $\omega_i=1$ for $i\in [n]$ with $i\ne1,3$.
	In this case, we will prove that $\overline{I^k}=I^k$ for all $k\ge 2$.
	Since $I^k\subseteq \overline{I^k}$ is always valid, it suffices to prove that  $\overline{I^k}\subseteq I^k$.
	
	Let $\xb^{\ab}=x_1^{a_1}\cdots x_n^{a_n} \in \mathcal{G}(\overline{I^k})$, then $v_\mathbf{a}^\ast(I)\geq k$ by Lemma \ref{normal}(\ref{normal-2}).  It follows from the definition of $v_\mathbf{a}^\ast(I)$ that there exists the
	vector $\mathbf{y}=(y_1,y_2,\dots,y_n)^T\in \mathbb{R}_{\geq 0}^{n}$ which satisfies the following  system (10) of  inequalities
	\[
	\begin{array}{cccc}
	(10)\quad&
	\left\{
	\begin{aligned}
	y_1+\dots+y_n &\geq k,&\circled{1}\\
	\omega_1y_1+y_n &\leq  a_1,&\circled{2}\\
	\omega_1y_1+y_2 &\leq a_2,&\circled{3}\\
	y_2+\omega_3y_3 &\leq a_3,&\circled{4}\\
	\omega_3y_3+y_4 &\leq a_4,&\circled{5}\\
	y_4+y_5 &\leq a_5,&\circled{6}\\
	\cdots\\
	y_{n-2}+y_{n-1} &\leq a_{n-1},\\
	y_{n-1}+y_n &\leq a_n.
	\end{aligned}
	\right.&\quad\quad(11)\quad&
	\left\{
	\begin{aligned}
	y_4 &\leq a_4,\\
	y_4+y_5 &\leq a_5,\\
	\vdots\\
	y_{n-2}+y_{n-1} &\leq a_{n-1},\\
	y_{n-1}+y_n &\leq a_n,\\
	y_n &\leq a_1.
	\end{aligned}
	\right.
	\end{array}
	\]
	Let $h=\lceil y_4+\dots+y_n\rceil$. Since the vector $\mathbf{y}=(y_4,y_5,\dots,y_{n})^T$  satisfies  the above  system (11) of  inequalities, it follows  from  Theorem \ref{divide3} that
 there are at least $h$ monomials  $f_{i1},\ldots,f_{ih}\in \{x_4x_5,\ldots,x_{n-1}x_n,x_nx_1\}$ such that $\xb^{\bb}$ can be  divisible by $\prod_{j=1}^{h}f_{ij}$, where ${\bb}=(a_1,0,0,a_4,\ldots,a_n)$. If
 $h\geq k$, then $\xb^{\bb}\in I^k$, which forces $\xb^{\ab}\in I^k$.

	In the following, we assume that $h<k$. One has
	\[
	a_2+a_3 \geq \omega_1y_1+2y_2+\omega_3y_3 \geq 2(y_1+y_2+y_3) \geq 2(k-h).
	\]
	We distinguish between the two cases:
	
	(i)  If $a_2, a_3 \geq k-h$, then  $(x_2x_3)^{ k-h} | x_2^{a_2}x_3^{a_3}$. Since $\xb^{\ab}=(x_2^{a_2}x_3^{a_3})\xb^{\bb}$ and $\xb^{\bb}$ can be  divisible by $\prod_{j=1}^{h}f_{ij}$, it follows  that
$\xb^{\ab}$  can be  divisible by $(x_2x_3)^{ k-h}\prod_{j=1}^{h}f_{ij}$, which implies that  $\xb^{\ab}\in I^k$.

(ii) If $a_2 \geq k-h>a_3$ or  $a_3 \geq k-h>a_2$. We can assume that $a_2 \geq k-h>a_3$ by  symmetry. In this case, we have
	$k-h>a_3\ge y_2+\omega_3y_3\ge  y_2+y_3$. Hence
	$y_1 \geq k-(y_2+y_3)-(y_4+\dots+y_n)\geq k-a_3-(y_4+\dots+y_n)\ge k-a_3-h$. It follows  that $a_1 \geq \omega_1y_1+y_n\ge \omega_1(k-a_3-h)+y_n$.
	On the other hand, $a_2+a_3\ge \omega_1y_1+2y_2+\omega_3y_3 \ge \omega_1y_1+2y_2+2y_3 \ge \omega_1(k-a_3-h)+\omega_1(y_1+a_3+h-k)+2(y_2+y_3) \ge \omega_1(k-h-a_3)+2a_3$.
This implies that $a_2 \geq \omega_1(k-h-a_3)+a_3$. Note that the vector $\mathbf{y}=(y_4,y_5,\dots,y_{n})^T$  also satisfies  the above  system (11) by  replacing  $a_1$ by  $a_1-\omega_1(k-h-a_3)$. it follows  from  Theorem \ref{divide3} that
 there are at least $h$ monomials  $g_{i1},\ldots,g_{ih}\in \{x_4x_5,\ldots,x_{n-1}x_n,x_nx_1\}$ such that $\xb^{\cb}$ can be  divisible by $\prod_{j=1}^{h}g_{ij}$, where ${\cb}=(a_1-\omega_1(k-h-a_3),0,0,a_4,\ldots,a_n)$. Therefore  $\xb^{\ab}$ can be  divisible by  $(x_1^{\omega_1}x_2^{\omega_1})^{k-h-a_3}(x_2x_3)^{a_3}(\prod_{j=1}^{h}g_{ij})$, so $\xb^{\ab} \in I^k$.
\end{proof}

\begin{Theorem}
\label{cycle3}
Let $C_\omega^n$ be a weighted cycle on the vertex set $[n]$, and let $I=I(C_\omega^n)$ be its edge ideal. If $I$ is integrally closed, then $I$ is normal.			
\end{Theorem}
\begin{proof}Let $E(C_\omega^n)=\{e_1,\dots,e_n\}$ with each  $\omega_i=\omega(e_i)$, where  $e_i=\{i,i+1\}$ for $i\in [n-1]$ and $e_n=\{1,n\}$.
Since $I$ is  integrally closed, $C_\omega^n$ has at most three edges with non-trivial weights by Theorem \ref{main}. If $C_\omega^n$ is a trivially weighted cycle, then $I$ is normal by
\cite[Proposition 2.1 and Corollary 2.8]{AVV} and \cite[Proposition 2.1.2]{HSV}.  Now we assume that $C_\omega^n$ has at least one edge with non-trivial weight. In this case, we will prove that $\overline{I^k}=I^k$ for all $k\ge 2$.
Since $I^k\subseteq \overline{I^k}$ is always valid, it suffices to prove that  $\overline{I^k}\subseteq I^k$.
We can divide this into the following two cases:

(1) If   $C_\omega^n$ has exactly three edges  or two edges with non-trivial weights, then $\overline{I^k}\subseteq I^k$ by Theorem \ref{cycle1} and Theorem \ref{cycle2}, respectively.

(2)  If $C_\omega^n$ has only one edge with non-trivial weight, then we can assume by symmetry that $\omega_1 \geq 2$ and $\omega_i=1$ for any $i=2,\ldots,n$.
In this case, let $\xb^{\ab}=x_1^{a_1}\cdots x_n^{a_n} \in \mathcal{G}(\overline{I^k})$, then $v_\mathbf{a}^\ast(I)\geq k$ by Lemma \ref{normal}(\ref{normal-2}).  It follows  from the definition of $v_\mathbf{a}^\ast(I)$ that there exists a vector $\mathbf{y}=(y_1,y_2,\dots,y_n)^T\in \mathbb{R}_{\geq 0}^{n}$  which satisfies the following  system (12)  of inequalities
\[
\begin{array}{cccc}
(12)\quad&
\left\{
\begin{aligned}
y_1+\dots+y_n &\geq k,&\circled{1}\\
\omega_1y_1+y_n &\leq  a_1,&\circled{2}\\
\omega_1y_1+y_2 &\leq a_2,&\circled{3}\\
y_2+y_3 &\leq a_3,&\circled{4}\\
y_3+y_4 &\leq a_4,&\circled{5}\\
y_4+y_5 &\leq a_5,&\circled{6}\\
\vdots\\
y_{n-2}+y_{n-1} &\leq a_{n-1},\\
y_{n-1}+y_n &\leq a_n.
\end{aligned}
\right.&\quad\quad(13)\quad&
\left\{
\begin{aligned}
y_2 &\leq a_2,\\
y_2+y_3 &\leq a_3,\\
y_3+y_4 &\leq a_4,\\
y_4+y_5 &\leq a_5,\\
\vdots\\
y_{n-2}+y_{n-1} &\leq a_{n-1},\\
y_{n-1}+y_n &\leq a_n,\\
y_n &\leq a_1.
\end{aligned}
\right.
\end{array}
\]
Let $h= \lceil y_2+\dots+y_n \rceil$. We distinguish between the following  two cases:

 (1) If $h \ge k$, then  the vector $\mathbf{y'}=(y_2,\dots,y_n)^T\in \mathbb{R}_{\geq 0}^{n-1}$ satisfies  the system (13) of inequalities. It follows from Theorem \ref{divide3} that  there exist at least $h$ monomials $f_{1},\ldots,f_{h}\in \{x_2x_3,\ldots,x_{n-1}x_n,x_nx_1\}$ such that $\xb^{\ab}$ is divisible by $\prod_{i=1}^{h}f_{i}$, so $\xb^{\ab} \in I^k$.

(2) If $h<k$, then $y_1\geq k-h$ by $\circled{1}$ in system (12). This implies that
\[
a_1 \geq \omega_1(k-h)+y_n \ \ \text{and} \ \ a_2 \geq \omega_1(k-h)+y_2.
\]
Thus   the vector $\mathbf{y'}=(y_2,\dots,y_n)^T\in \mathbb{R}_{\geq 0}^{n-1}$ satisfies the system (13) by  replacing $a_1$ by $a_1-\omega_1(k-h)$ and $a_2$ by $a_2-\omega_1(k-h)$.
By Theorem \ref{divide3}, there exist at least $h$ monomials $g_{1},\ldots,g_{h}\in \{x_2x_3,\ldots,x_{n-1}x_n,x_nx_1\}$ such that $\xb^{\bb}$ is divisible by $\prod_{i=1}^{h}g_{i}$, where $\mathbf{b}=(a_1-\omega_1(k-h),a_2-\omega_1(k-h),a_3,\ldots,a_n)$. So $\xb^{\ab}$ can be  divisible by $(x_1^{\omega_1}x_2^{\omega_1})^{(k-h)}\prod_{i=1}^{h}f_{i}$. Therefore,
$x^\mathbf{a} \in I^k$.
\end{proof}

\begin{Theorem}
\label{line}
Let $L_\omega^n$ be a weighted path on the set $[n]$, and let $I=I(L_\omega^n)$ be its edge ideal. If $I$ is integrally closed, then $I$ is normal.
\end{Theorem}
\begin{proof}
Let $E(L_\omega^n)=\{e_1,\ldots,e_{n-1}\}$  with  each $\omega_i=\omega(e_i)$ and $e_i=\{i,i+1\}$. Since  $L_\omega^n$ is an induced subgraph of $C_\omega^{n+1}$, where $E(C_\omega^n)=E(L_\omega^n)\cup \{e_n,e_{n+1}\}$, $e_n=\{n,n+1\}$, $e_{n+1}=\{1,n+1\}$
 and $\omega(e_n)=\omega(e_{n+1})=1$.  Since $I$ is integrally closed,  $I(C_\omega^{n+1})$ is also integrally closed by Theorem \ref{main}. Again by Theorem \ref{cycle3}, we get that $I(C_\omega^{n+1})$ is normal, which implies that $I$ is normal by Remark \ref{induced graph}.
\end{proof}

\medskip
\hspace{-6mm} {\bf Acknowledgments}

 \vspace{3mm}
\hspace{-6mm}  This research is supported by the Natural Science Foundation of Jiangsu Province (No. BK20221353). The authors are grateful to the computer algebra systems Normaliz \cite{BRTS} for providing us with a large number of examples.

\medskip
\hspace{-6mm} {\bf Data availability statement}

\vspace{3mm}
\hspace{-6mm}  The data used to support the findings of this study are included within the article.

\medskip
\hspace{-6mm} {\bf Conflict of interest statement}

\vspace{3mm}
\hspace{-6mm}  All authors declare that they have no conflicts of interest to this work.

\end{document}